\documentclass[10pt,oneside]{amsart}
\usepackage{amssymb, amscd, amsmath, amsthm, latexsym, hyperref}
\usepackage{pb-diagram, fancyhdr, graphicx, psfrag, enumerate}
\usepackage[text={6.3in,8.5in},centering,letterpaper,dvips]{geometry}

\newcommand{\Alg}{\alg}
\newcommand{\alg}{\operatorname{alg}}
\newcommand{\Alex}{\operatorname{Alex}}
 \newcommand{\smMatrix}[1]{\left[\begin{smallmatrix}#1\end{smallmatrix}\right]}
 \newcommand{\FC}{\mathcal{F}}
 \newcommand{\GC}{\mathcal{G}}
\newcommand{\Cone}{\operatorname{Cone}}
\newcommand{\cfkinv}{{\rm CFKI}^\infty(K)}

\def\ff{{\mathbb F}}

\def\co{\colon\!}
\def\cs{\mathop{\#}}

\def\cala{\mathcal{A}}

\def\calt{\mathcal{T}}

\DeclareMathOperator{\cfk}{\rm CFK}
\DeclareMathOperator{\icfk}{\rm CFKI}
\DeclareMathOperator{\ihfk}{\it \rm HFKI}
\DeclareMathOperator{\hfk}{\it \rm HFK}

\DeclareMathOperator{\Max}{\operatorname{Max}}
\DeclareMathOperator{\Min}{\operatorname{Min}}

\newcommand{\spinc}{\ifmmode{{\mathfrak s}}\else{${\mathfrak s}$\ }\fi}
\newcommand{\spinct}{\ifmmode{{\mathfrak t}}\else{${\mathfrak t}$\ }\fi}
\newcommand{\spincw}{\ifmmode{{\mathfrak w}}\else{${\mathfrak w}$\ }\fi}
\def\Z{\mathbb Z}

\def\R{\mathbb R}

\def\F{\mathbb F}

\newcommand{\inv}{^{-1}}

\def\cfki{\cfk^\infty(K)}
\def\hfki{\mathrm{HFKI}}
\def\com{{\rm C}}
\def\comd{{\rm D}}
\newcommand{\op}{\operatorname{op}}

\newcommand{\fig}[2] { \includegraphics[scale=#1]{#2} }

 \def\involmap{\mathcal{I}}

\newtheorem{theorem}{Theorem}

\newtheorem{proposition}[theorem]{Proposition}

\theoremstyle{definition}
\newtheorem{definition}[theorem]{Definition}

\newtheorem{notation}[theorem]{Notation}

\begin{document}
\title[Involutive Heegaard Floer Knot Homology]{An Involutive Upsilon Knot Invariant}
\author{Matthew Hogancamp}
\author{Charles Livingston}
\address{Matt Hogancamp: Department of Mathematics, University of Southern California, Los Angeles, CA  90089}
\email{hogancam@usc.edu}
\address{Charles Livingston: Department of Mathematics, Indiana University, Bloomington, IN 47405 }
\email{livingst@indiana.edu}
\thanks{This work  was supported by a grant from the Simons Foundation and by the National Science Foundation under grant  DMS-1505586.}

\begin{abstract}  Using the theory of involutive Heegaard Floer knot theory developed by Hendricks-Manolescu, we define two involutive analogs of the Upsilon knot concordance invariant of Ozsv\'ath-Stipsicz-Szab\'o.  These involutive   invariants are  piecewise linear functions defined on the interval [0,2].  Each is a concordance invariant and  provides bounds on the three-genus of a knot.  
\end{abstract}

\maketitle
\section{Introduction}
Heegaard Floer knot theory~\cite{os_knotinvars} associates to a knot $K \subset S^3$ a chain complex $\cfk^\infty(K)$.  To be more precise, $$\cfk^\infty(K) = (\mathrm C, \partial,  \alg, \Alex),$$
where $\com$ is a  graded $\Z_2$--chain complex with boundary map $\partial$ of degree $-1$ and $\alg$ and $\Alex$ are increasing filtrations on $(\mathrm{C}, \partial)$.  Furthermore, $\mathrm C$ is a free, finitely generated $  \Z_2[U,U^{-1}]$--module; the  action by $U$ commutes with $\partial$, lowers gradings by $2$ and lowers both filtration levels by 1.  The construction of $\cfk^\infty(K)$ depends on a series of choices, but any two complexes  associated to $K$ are bifiltered chain homotopy equivalent.  

Two further structural properties of $\cfk^\infty(K)$ have  been discovered: Sarkar~\cite{sarkar} described a naturally defined self-chain homotopy equivalence $s$, now called the {\it Sarkar map}, and  Hendricks and Manolescu~\cite{hm}  used the existence of  a \emph{skew-bifiltered} chain homotopy equivalence, first constructed in~\cite{os_knotinvars}, $$\involmap \co  (\mathrm C, \partial ) \to  (\mathrm C, \partial ),$$  to define a family of new invariants called  the involutive homology groups,  $\ihfk^\circ(K)$, where $\circ =\infty,  +, -, $ or $\widehat{\ \ }$.  Here, by {\it skew-bifiltered}  we mean that $\involmap$ switches algebraic and Alexander filtration levels.  That  $\ihfk^\circ(K)$  is well-defined depends on the naturality of $\involmap$, which follows from results of Juhasz-Thurston~\cite{jt}.  Hendricks and Manolescu also proved that $\involmap^2 $ is bifiltered chain homotopic to $ s$.  

The group $\ihfk^\circ(K)$ is the homology of the mapping cone of $\involmap + I$.  
In Section~\ref{sec:involve}    we will review the construction of this mapping cone,  $\icfk^\infty(K)$. We will also  describe a bifiltration on $\icfk^\infty(K)$.  Section~\ref{sec:reduce}  describes  the computation of $\icfk^\infty(K)$ for the torus knot $K=T(3,7)$  and presents a generalization that applies to all torus knots  or, more generally, to $L$--space knots and their mirror images.  In Section~\ref{sec:upsilon} we describe how the Upsilon function associated to $\cfk^\infty(K)$, defined in~\cite{oss}, can be extended to give a pair of what we call the  {\it involutive Upsilon} functions.   Section~\ref{sec:t37} focuses on a single example, the torus knot $T(3,7)$.  In Section~\ref{sec:concordance} we prove the concordance invariance of the involutive  Upsilon functions.   We show in Section~\ref{sec:v0} that the value of each Upsilon at $t=0$ is determined by previously defined invariants, $\overline{V}_0 (K)$ and  $\underline{V}_0 (K)$.  Section~\ref{sec:three-genus} briefly discusses the three-genus of knots.

\subsection{Conventions}
Throughout this paper, all complexes are $\Z$--graded chain complexes over the field $\Z/2$, henceforth denoted $\F$.  All differentials will lower homological degree by 1.  The word \emph{grading} is synonymous with \emph{homological degree}.

Recall that if $(I,\leq)$ is a partially ordered set, then an $I$--filtered complex is a complex $\com$ together with a collection of subcomplexes $\mathcal{F}_i(\com)\subset \com$, indexed by $i\in I$, with $\mathcal{F}_i(\com)\subseteq \mathcal{F}_j(\com)$ whenever $i\leq j$.   We will always assume that
\begin{equation}\label{eq:filtrationCondition}
\bigcap_i \FC_i(\com)=\{0\} \ \ \ \ \ \ \ \ \text{ and } \ \ \ \ \ \ \ \ \ \  \bigcup_i\FC_i(\com) = \com.
\end{equation}
A homotopy equivalence between $I$--filtered complexes $\com\rightarrow \com'$ is said to be an $I$--filtered homotopy equivalence if all relevant maps (the chain map  
$\com \to  \com'$, its chain homotopy inverse $\com' \to \com$, and the chain homotopy) are filtration preserving.

By default, \emph{filtered complex} means $\Z$--filtered and \emph{bifiltered} means $\Z\times\Z$--filtered.  If $(\com,\FC)$ is a filtered complex, we may define the \emph{filtration degree} of a nonzero $x\in \com$ by
\begin{equation}\label{eq:filtrationDegree}
\deg_\FC(x) = \min\{i \ | \ x\in \FC_i(\com)\}.
\end{equation}
By convention, we set $\deg_\FC(0)=-\infty$.   Thus, given the assumptions \eqref{eq:filtrationCondition}, $\deg_\FC$ is a well-defined set function $\com\rightarrow \Z\cup\{-\infty\}$.  The function $\deg_\FC$ satisfies
\begin{itemize}
\item[(F1)] $\deg_\FC\inv(-\infty)=\{0\}$,
\item[(F2)] $\deg_\FC(x+y)\leq \max\{\deg_\FC(x),\deg_\FC(y)\}$, and
\item[(F3)] $\deg_\FC(\partial(x))\leq \deg_\FC(x)$.
\end{itemize} 
Conversely, given a set function with these properties, we can recover the subcomplex $\FC_i(\com)$ as the linear span of all elements $x\in \com$ such that $\deg_\FC(x)\leq i$.  In fact, we can allow more general functions $\deg_\FC$.  If $\deg_\FC\co \com \rightarrow \R$ is a set function satisfying the properties (F1), (F2), and (F3), then we can define an $\R$--filtration of $\com$, which is the associated ordered family of subcomplexes $\FC_t(\com)\subset \com$, indexed by $t\in \R$.  In all examples of interest to us, the image of $\com\setminus\{0\}$ under $\deg_\FC$ will be a discrete subset $S\subset \R$ which, as an ordered set,  is isomorphic to  $\Z$.  Thus, every $\R$--filtered complex considered here can also regarded as a $\Z$--filtered complex, though the precise description would require choosing an isomorphism $S\cong \Z$.

\begin{notation}
Henceforth, a filtered complex will mean a pair $(\com,\deg_\FC)$, where $\com$ is a chain complex with  differential $\partial$, and $\deg_\FC$ is a function $\com\rightarrow \R\cup \{-\infty\}$ satisfying properties (F1), (F2), and (F3) above.  Similarly, a \emph{bifiltered complex} is a triple $(\com,\deg_\FC, \deg_\GC)$ such that $(\com,\deg_\FC)$ and $(\com,\deg_\GC)$ are filtered complexes.  The reader can verify that a bifiltered complex in this sense corresponds to a filtration indexed by $\Z \times \Z$.
\end{notation}
\vskip.05in

\noindent{\it Acknowledgments}  We thank Kristen Hendricks and Jen Hom for helpful comments.

\section{Involutive homology}\label{sec:involve}
We begin by reviewing the definition of the mapping cone complex in the  context of the chain map $\involmap  + I$; we denote this complex $\Cone(\mathrm{\com}, \involmap +I)$.  The underlying graded vector space  is $\cfki[1] \oplus \cfki$, where $\cfki[1] $ is the same complex as $\cfki$ with gradings shifted up by 1.  The  boundary map is given by 
$$
\partial_\involmap =
\smMatrix{\partial_\com & 0 \\ \involmap+I & \partial} \colon \ \cfki[1] \oplus \cfki  \longrightarrow \cfki[1]\oplus \cfki.
$$
In other words, the boundary of $(z_1,z_2)$ is $(\partial(z_1),\involmap(z_1) + z_1+ \partial(z_2))$.  It is easily checked that  $(\partial_\involmap)^2 = 0$. 

\begin{definition}  We denote  $\Cone(\com, \involmap +I)$ by $\icfk^\infty(K)$.\end{definition}

\begin{theorem}The homology of  $\icfk^\infty(K)$, $\ihfk^\infty(K)$,  is isomorphic as an $\F[U,U^{-1}]$--module to $\ff[U,U^{-1}] \oplus   \ff[U,U^{-1}]$, where $(1,0)$ has grading 1 and $(0,1)$ has grading $0$.  This splitting is natural. 
\end{theorem}
\begin{proof}  For the mapping cone of any map of complexes $f\co \com \to \comd$, there is a long exact sequence $$\cdots \to {\rm H}_i(\com) \xrightarrow{f_*} H_i(\comd) \to  {\rm H}_i(\text{Cone}(\com,\comd,f)) \to  {\rm H}_{i-1}(\com)  \xrightarrow{f_*} \cdots .$$  The map $(\involmap + I)_*$ is trivial on homology and $ {\rm H}_i(\com) \cong  {\rm H}_i(\comd) \cong \F[U, U^{-1}]$ with the generator 1 of grading 0.   
\end{proof}

\begin{definition}  $\ihfk^\infty(K) \cong \calt_{o} \oplus \calt_{e}$, 
where $\calt_{o}$ is the {\it odd tower} isomorphic to $\F[U,U^{-1}] $ with all gradings odd,  and
$\calt_{e} $ is similarly the {\it even tower}.
\end{definition}

\subsection{The folded bifiltration}  
The map $\involmap$ is skew; it does not preserve the algebraic-Alexander bifiltration.  However, there is  a pair of filtrations on $\icfk^\infty(K) $ that are preserved.  

\begin{definition}\label{def:folding}
Suppose $(\com,\deg_\FC,\deg_\GC)$ is a bifiltered complex. Define a new bifiltered complex $(\com,\Min, \Max)$, where $\Min(x)= \min\{\deg_\FC(x),\deg_\GC(x)\}$ and $\Max(x)=\max\{\deg_\FC(x),\deg_\GC(x)\}$.  We call the resulting bifiltration the \emph{folded}, or \emph{Min-Max bifiltration on $\com$}.
\end{definition}

We may regard $\cfki $ as a bifiltered complex with respect to the Min-Max filtration.  Note that $\involmap\colon \cfki\rightarrow \cfki$ swaps the $\Alg$ and $\Alex$ filtrations, and hence it preserves the Min-Max filtrations.  Thus, $(\cfkinv, \Min, \Max)$ is a bifiltered complex.

The left diagram in Figure~\ref{fig37invol2} illustrates a {\it model} complex for   $\cfk^\infty(T(3,7))$; in general, the structure of $\cfk^\infty(T(p,q))$ is determined by the results of~\cite{os_l_space}, which study a more general family of knots, called {\it $L$--space knots}. In the diagram, the vertex at $(0,6)$ represents a generator at grading 0.  The full complex is constructed from the illustrated model complex by tensoring with $\F[U,U^{-1}]$; the $U^k$ translates, if illustrated, would be represented by copies of the finite complex that is drawn, shifted $-k$ units along the main diagonal.  It is evident that the only self-chain homotopy equivalence that is also skew is represented by reflection through the diagonal.   On the right in Figure~\ref{fig37invol2} the complex $\cfk^\infty(T(3,7))$ is illustrated, where now the  bifiltration is given by   Min-Max.   Figure \ref{xyz} illustrates the involutive complex $\cfkinv$ for $T(3,7)$  as well as an equivalent reduced complex obtained by {\it bifiltered Gaussian elimination}.  In the next section we describe the steps in constructing this reduction.

\begin{figure}[h]
\fig{.4}{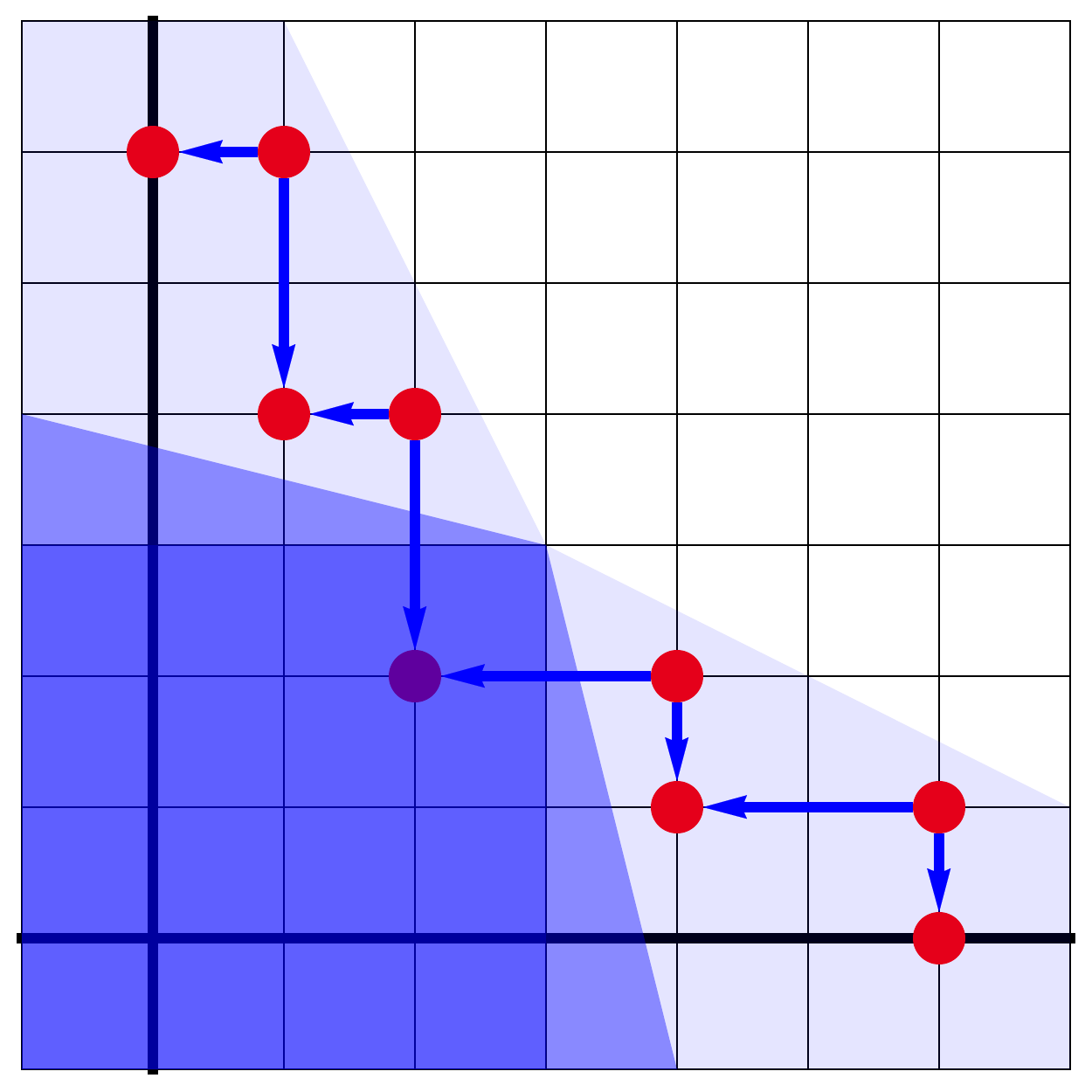} \hskip.4in  \fig{.545}{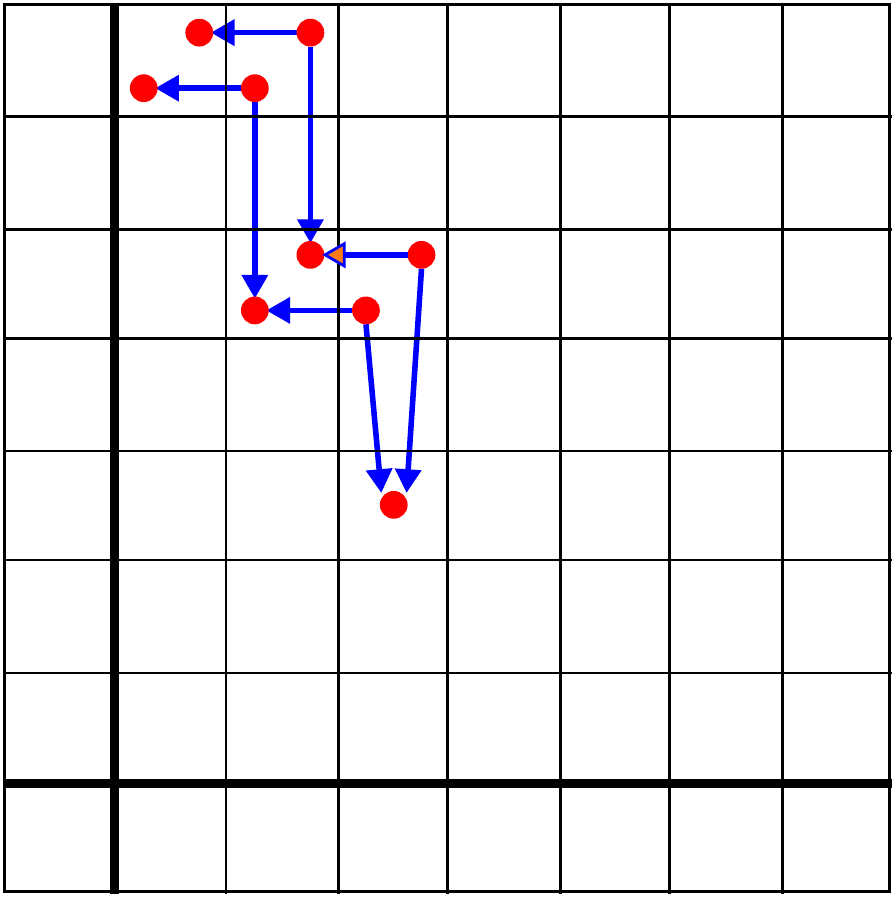}
\caption{$\cfk^\infty(T(3,7))$ and folded $\cfk^\infty(T(3,7))$}
\label{fig37invol2}
\end{figure}

\begin{figure}[h]\fig{.4}{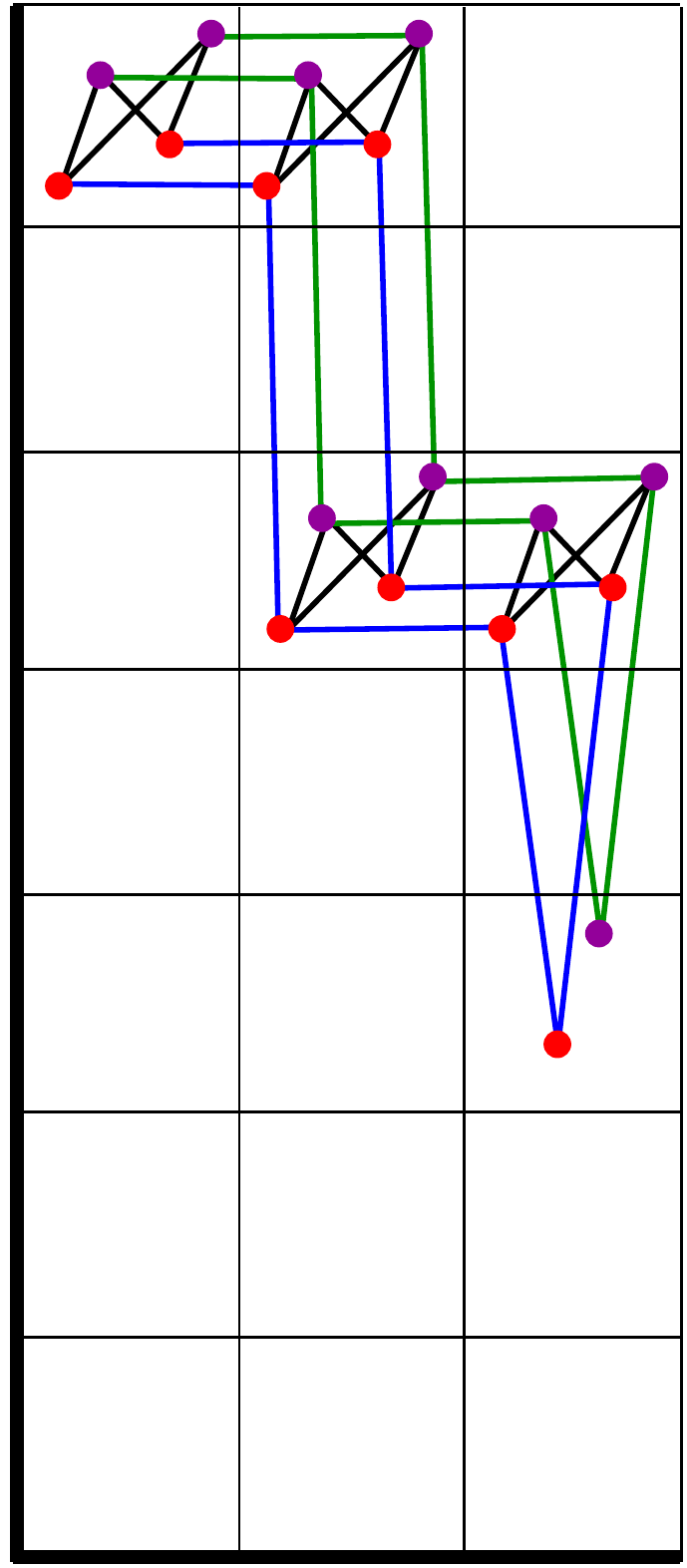} \hskip.5in \fig{.4}{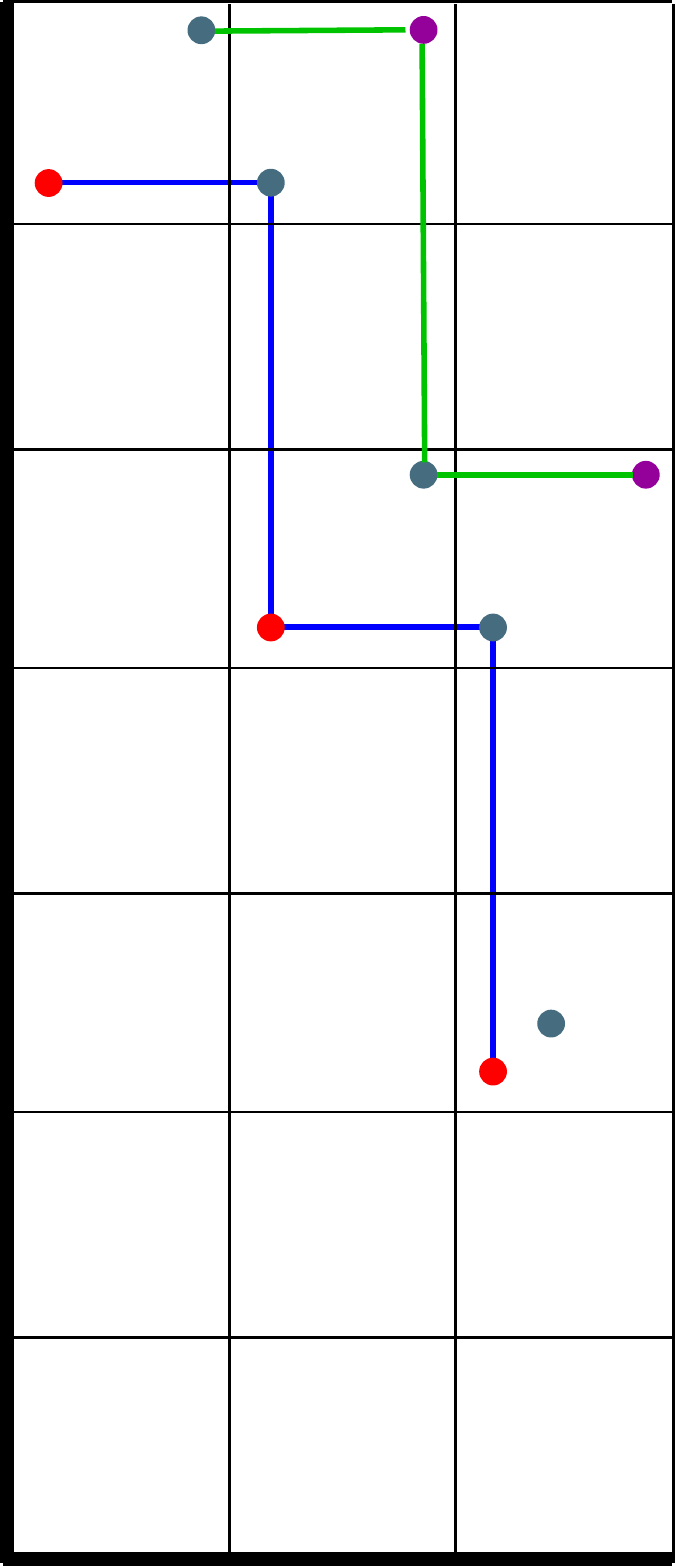}\caption{$\icfk(T(3,7))$ and  its reduction.}\label{xyz} \end{figure}

\section{Reductions}\label{sec:reduce}

Suppose $(I,\leq)$ is a partially ordered set, and let $(\com,\FC)$ be a filtered complex.  We say that $\com$ is \emph{reduced} if each subquotient $\FC_{\leq i}(\com)/\FC_{<i}(\com)$ has zero differential.  Under mild finiteness assumptions on $I$ (which are satsified by $I=\Z$ and $I=\Z\times \Z$), any $I$--filtered complex is isomorphic to $\com_{\text{red}}\oplus Z$ where $\com_{\text{red}}$ is reduced and $Z$ is homotopically trivial in the filtered sense.  An algorithm for reducing bifiltered complexes is presented in~\cite{rasmussen}.  In this section we will summarize the procedure in the case that the starting complex is the involutive complex associated to a staircase, such as the one illustrated on  the left of Figure~\ref{fig37invol2}. In this example there are 9 vertices  and the steps are $[1,2,1,2,2,1,2,1]$.  The goal is to perform a bifiltered change of basis so that, after removing acyclic summands, the remaining diagram has no arrows within any given square.    We first introduce some terminology.

\begin{definition} A staircase complex $\com$ is \emph{symmetric} if $\com\simeq \com^{\text{op}}$, where $\com^{\op}$ denotes the staircase with bifiltrations swapped.  Any symmetric staircase $\com$ has steps of size $[a_1,\ldots,a_{2k}]$ with $a_i=a_{2k-i+1}$.  A symmetric staircase is \emph{inward-pointing} if the generator corresponding to the central dot is a cycle, and is \emph{outward-pointing} otherwise.  We call a staircase \emph{positive} if the first (top) step is to the right, not down, and \emph{negative} otherwise.  
\end{definition}

Note that every symmetric staircase has an even number of edges.  A positive  symmetric staircase is inward-pointing exactly when its length is 0 (mod 4), and a negative  symmetric staircase is inward-pointing exactly when its length is 2 (mod 4).  

\vskip.05in
\noindent {\bf Note:}  Positive torus knots,  $T(p,q)$ with $p, q >0$, have positive  symmetric  staircase complexes, while their mirror images, $-T(p,q)$, have negative staircase complexes.  More generally, all $L$--space knots have positive staircase complexes.\vskip.05in  

The appropriate change of basis is best described by a schematic, as in Figure~\ref{schema}.  
\begin{figure}[h]
\fig{.7}{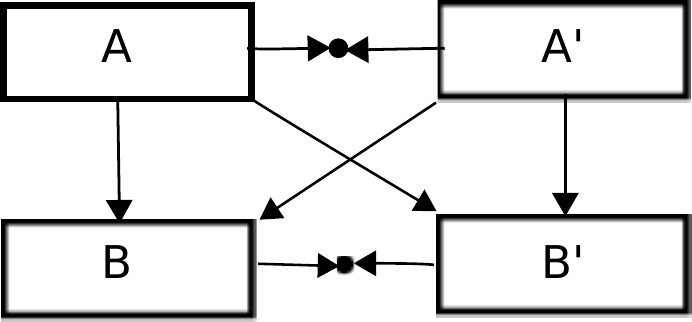}\hskip.7in  \fig{.7}{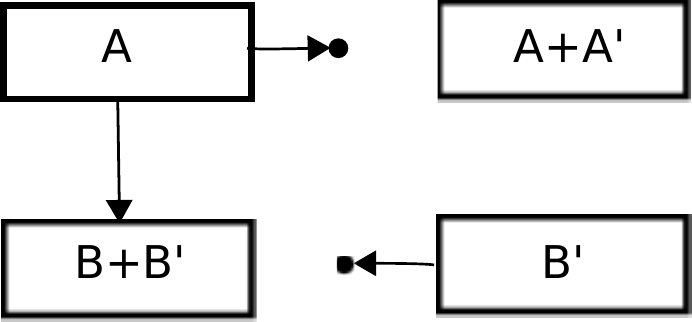}
\caption{Schematics for  $\icfk^\infty(T(3,7))$}
\label{schema}
\end{figure}
The diagram on the left represents the complex $\icfk^\infty(T(3,7))$.  The box labeled $B$ is the portion of the staircase complex that arises from the portion of the staircase with $i <j$, and $B'$ is the portion with $i >j$.  The boxes $A$ and $A'$ represent the same complexes, with grading shifted up by one.

A change of basis in which generators from $A$ are added to corresponding generators of $A'$, and generators of $B'$ are added to corresponding generators of $B$, changes the diagram so that it appears as on the right in Figure~\ref{schema}. 
The complexes $A + A'$ and $\bullet  \leftarrow B'$ are each staircase complexes; the one with an  even number of vertices is acyclic and the other has homology of rank one.   (In the current example, $A+A'$ and $B$ both have four vertices, so it is $\bullet  \leftarrow B'$ that is not acyclic.)
The complex involving $A$ and $B + B'$ is illustrated schematically on the left in Figure~\ref{schema2}.  Changing basis, adding some of the generators from the top row to those of the bottom row, as indicated in the diagram, offers the reduction as illustrated on the right in Figure~\ref{schema2}.

\begin{figure}[h]
\fig{.7}{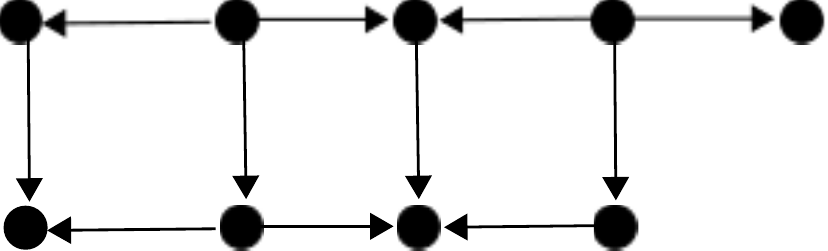}\hskip.7in  \fig{.7}{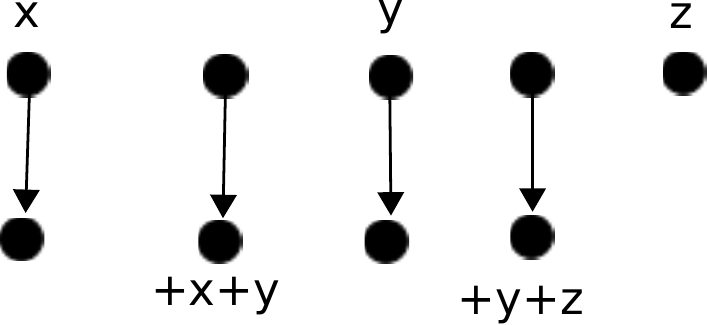}
\caption{Schematic reduction of complex.}
\label{schema2}\end{figure} 

We now see that $\icfk^\infty(T(3,7))$ is bifiltered chain homotopy equivalent to the complex illustrated in Figure~\ref{xyz}.  In the diagram, the staircase with four vertices is acyclic, and thus it doesn't contribute to later computations.  Henceforth, we will not include such acyclic pieces in our diagrams.

In considering a general symmetric staircase complex, there are two parity issues: the length of the staircase (mod 4) and whether the staircase is positive or negative.  The following theorem summarizes the result in all cases.  The proof in each case follows the exact same lines as the computations above.   

\begin{theorem}  Let $\mathrm{\com}$ be a staircase complex $[a_1, a_2, \ldots a_k, a_k, \ldots , a_1]$,  let $s = \sum a_i$ and let $d = \sum_{i \text{\ odd\ },\  i \le k} |a_i|$.  The involutive complex associated to $\mathrm{C}$ is the direct sum of two complexes, one represented by a single vertex $v_0$ on the diagonal and the other a staircase complex $\mathrm{S}$.  

\begin{itemize} 
\item If the staircase $\mathrm{C}$ is positive and $k$ is even, then $v_0$ has grading level 1 and is at filtration level $(d,d)$;  $\mathrm{S}$ is the staircase $[a_1, \ldots, a_k]$ with homology at grading 0, beginning at   filtration level  $(0,s)$ and ending at $(d,d)$. \vskip.05in    

\item If the staircase $\mathrm{C}$ is positive and $k$ is odd, then $v_0$ has grading level 1 and is at filtration level $(d,d)$;  $\mathrm{S}$ is the staircase $[a_{1}, \ldots, a_{k-1}]$ with homology at grading 0, beginning at filtration level $(0,s)$.\vskip.05in   

\item If the staircase $\mathrm{C}$ is negative and $k$ is even, then $v_0$ has grading level 0 and is at filtration level $(-d,-d)$;  $\mathrm{S}$ is the staircase $[a_1, \ldots, a_k]$ with homology at grading 1, beginning at the  filtration level  $(-s,0)$ and ending at $(-d,-d)$. \vskip.05in

\item If the staircase $\mathrm{C}$ is negative and $k$ is odd, then $v_0$ has grading level 0 and is at filtration level $(-d,-d)$;  $\mathrm{S}$ is the staircase $[a_{1}, \ldots, a_{k-1}]$ with homology at grading 1, beginning at filtration level $(-s,0)$.\vskip.05in

\end{itemize}

\end{theorem}  

Figures~ \ref{fig25invol2}, \ref{fig27invol2},  \ref{figneg25invol2}  and \ref{figneg27invol2} illustrate each case.  (In the colored version of this paper,  the green, red, gray, and purple  dots represent generators of homological degree $-1$,  $0, 1 $ and $2$, respectively.) 

\begin{figure}[h]
\fig{.4}{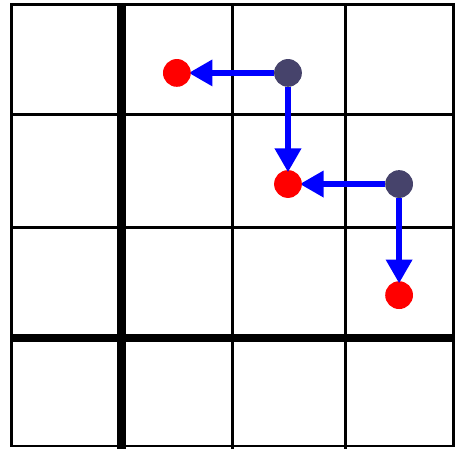}  \hskip.2in \fig{.4}{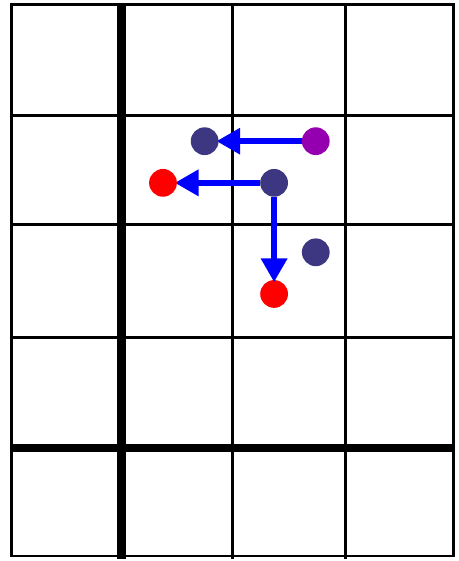}  
\caption{$\cfk^\infty(T(2,5)) $ and $  \icfk^\infty(T(2,5))$: Postive, Even}
\label{fig25invol2}
\end{figure}
\begin{figure}[h]
\fig{.4}{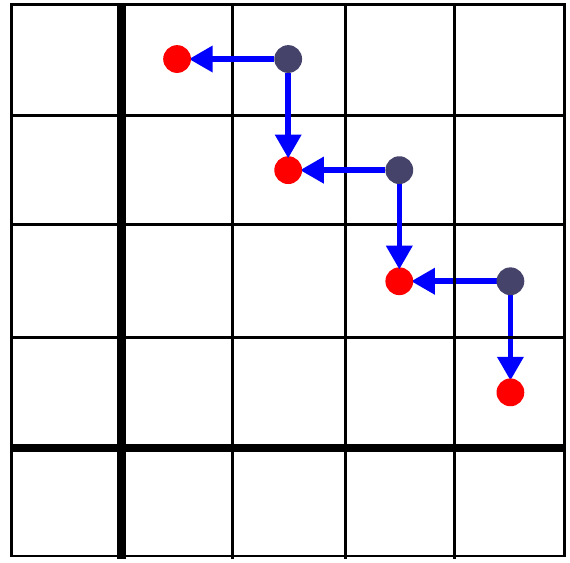}  \hskip.2in \fig{.4}{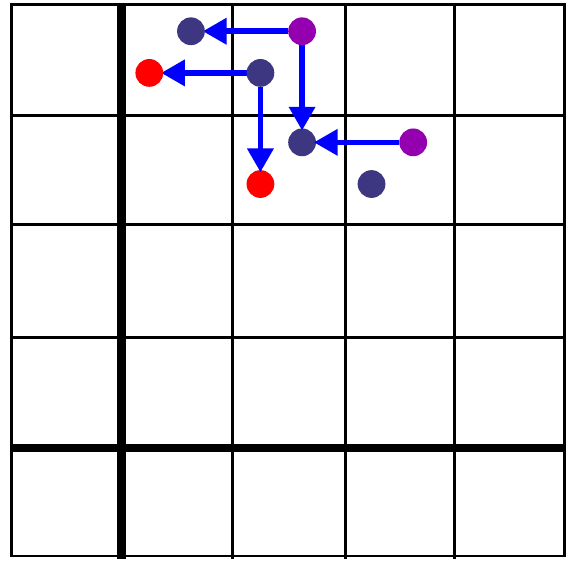}  
\caption{$\cfk^\infty(T(2,7)) $ and $  \icfk^\infty(T(2,7))$:   Postive, Odd}
\label{fig27invol2}
\end{figure}

\begin{figure}[h]
\fig{.4}{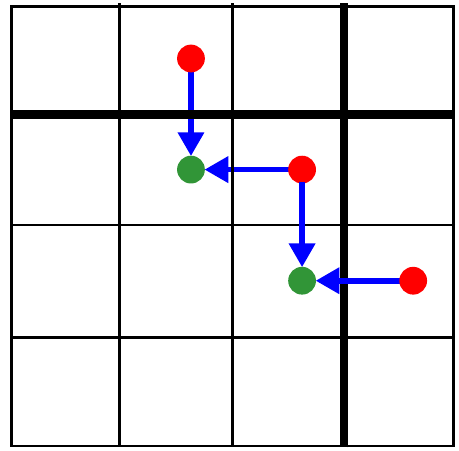}  \hskip.2in \fig{.4}{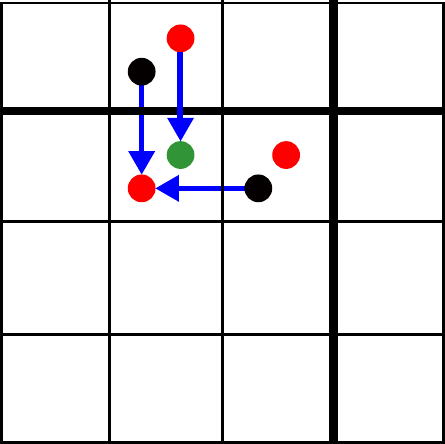}  
\caption{$\cfk^\infty(-T(2,5))$ and $\icfk^\infty(-T(2,5))$: Negative, Even}
\label{figneg25invol2}
\end{figure}

\begin{figure}[h]
\fig{.4}{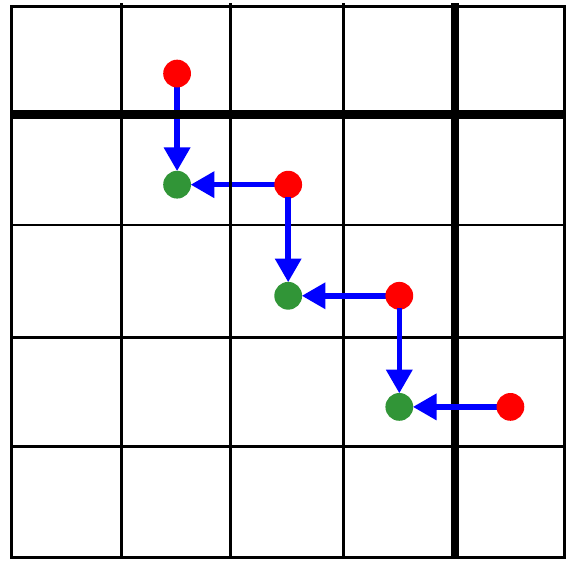}  \hskip.2in \fig{.4}{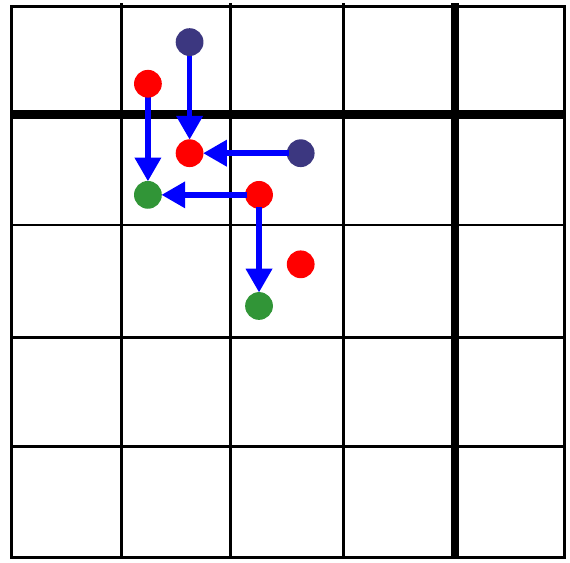}  
\caption{$\cfk^\infty(-T(2,7))$ and $\icfk^\infty(-T(2,7))$: Negative, Odd}
\label{figneg27invol2}
\end{figure}

\section{Upsilon}\label{sec:upsilon}  

In summary we have the following general result.

\begin{theorem} The complexes $\cfki$ and  $\icfk^\infty(K)$ are bifiltered by $\Max$ and $\Min$.   The homology group $H_*(\cfk^\infty(K))$ is isomorphic to a tower   $\calt_{E}$ of   even homological grading.
The homology group $H_*(\icfk^\infty(K))$ is isomorphic to a direct sum of towers $\calt_{O}$ and $\calt_{E}$ of odd and even homological grading, respectively.
\end{theorem}

\begin{definition}\label{def:nu}
For each $t \in [0,2]$, let  $\deg_t$ be the function on either $\cfk^\infty(K)$ or $\icfk^\infty(K)$ given by $$\deg_t(x) = \frac{t}{2} \Max(x) + (1-\frac{t}{2} )\Min(x).$$
The image of $\deg_t(x)$ as $x$ ranges over all (nonzero) elements of $\cfkinv$ is some discrete subset $S\subset\R$; for each $s$ we have a subcomplex $\cfkinv_{\deg_t\leq s}\subset \cfkinv$  which is the $\F$--span of all elements $x$  for which  $\deg_t(x) \le s$.  We refer to this as the slope $1-\frac{t}{2}$ filtration of $\cfkinv$ (or of $\cfk^\infty(K)$).   Note that for all $t$ and $s$, $ \cfkinv_{\deg_t\leq s}$ is an $\ff[U]$--module.  Similar statements hold for $\cfk^\infty(K)$.  Note also that $\deg_t$ induces  filtrations on $\hfk^\infty(K) $ and $\hfki^\infty(K)$.

\end{definition}

Recall that we have a short exact sequence of bifiltered complexes
$$
0\rightarrow \cfk^\infty(K) \buildrel \sigma\over \rightarrow \cfkinv \buildrel \pi\over \rightarrow \cfk^\infty(K)[\text{1}]\rightarrow 0.
$$
The connecting homomorphism for the associated long exact sequence is induced by the map $\involmap  + I$, and since $\involmap$ induces the identity map in homology, the connecting map is zero.  Thus the long exact sequence in homology splits into a collection of short exact sequences
\[
0\rightarrow \hfk^\infty_i(K)\buildrel \sigma_\ast\over \rightarrow \hfki^\infty_i(K) \buildrel \pi_\ast\over \rightarrow \hfk^\infty_{i-1}(K)\rightarrow 0
\]
for all $i$.  Since $\hfk^\infty(K)$ is supported in even homological degrees, it follows that $\sigma_\ast$ is an isomorphism $ \hfk^\infty_{0}(K) \to \hfki^\infty_{0}(K)$ and $\pi_\ast$ is an isomorphism $\hfki^\infty_{1}(K)\to \hfk^\infty_{0}(K)$.

Let $c_0 \in \hfki_0(K)$,  $c_1 \in \hfki^\infty_1(K)$, and $d \in \hfk^\infty_0(K)$ denote generators.   It follows  that $c_0=\sigma_\ast(d)$ and $d=\pi_\ast(c_1)$.  We now define three   functions of $t$:

$$
\nu^f_t(K) := \deg_t(d)
$$
$$
\overline{\nu}_t(K) = \deg_t(c_0)
$$
$$
\underline{\nu}_t(K) = \deg_t(c_1)
$$
Here the superscript on $\nu^f_t(K)$   indicates that we are using the folded bifiltration on $\hfk^\infty(K)$. Using these, we define two involutive Upsilon functions and a folded Upsilon function.

\begin{definition} $ \ $ 
\begin{center}$ \overline{\Upsilon}_t(K) = -2 \overline{\nu}_t(K)$ \hskip.5in
$ {\Upsilon}^f_t(K) = -2  {\nu}^f_t(K)$ \hskip.5in
$ \underline{\Upsilon}_t(K)  = -2   \underline{\nu}_t(K) $
\end{center}
\end{definition}

Our conventions for ``upper bar'' and ``lower bar'' where chosen so that the following inequalities hold:
\begin{proposition}\label{prop:upsilonOrder}

$\underline{\Upsilon}_t(K) \le  
{\Upsilon}^f_t(K)  \le 
\overline{\Upsilon}_t(K) .    $

\end{proposition}\label{prop:inequal}
\begin{proof}
If $(\com,\deg_\FC)$ and $(\comd,\deg_\FC)$ are filtered vector spaces, then any filtration preserving map $g\co\thinspace \com\rightarrow \comd$ satisfies $\deg_\FC(g(x))\leq \deg_\FC(x)$ by definition.  Thus, the inequalities $\overline{\nu}_t(K)\leq \nu^f_t(K)\leq \underline{\nu}_t(K)$ follow immediately from the fact that $c_0=\sigma_\ast(d)$ and $d=\pi_\ast(c_1)$, discussed in the remarks preceding the proposition.  This proves $\underline{\Upsilon}_t(K) \le   {\Upsilon}^f_t(K)  \le   \overline{\Upsilon}_t(K)$, as claimed.
\end{proof}

\section{Example: $T(3,7)$}\label{sec:t37}  In Figure~\ref{up37} we have redrawn the fully reduced complex $\icfk^\infty(T(3,7))$.
\begin{figure}[h]
\fig{.35}{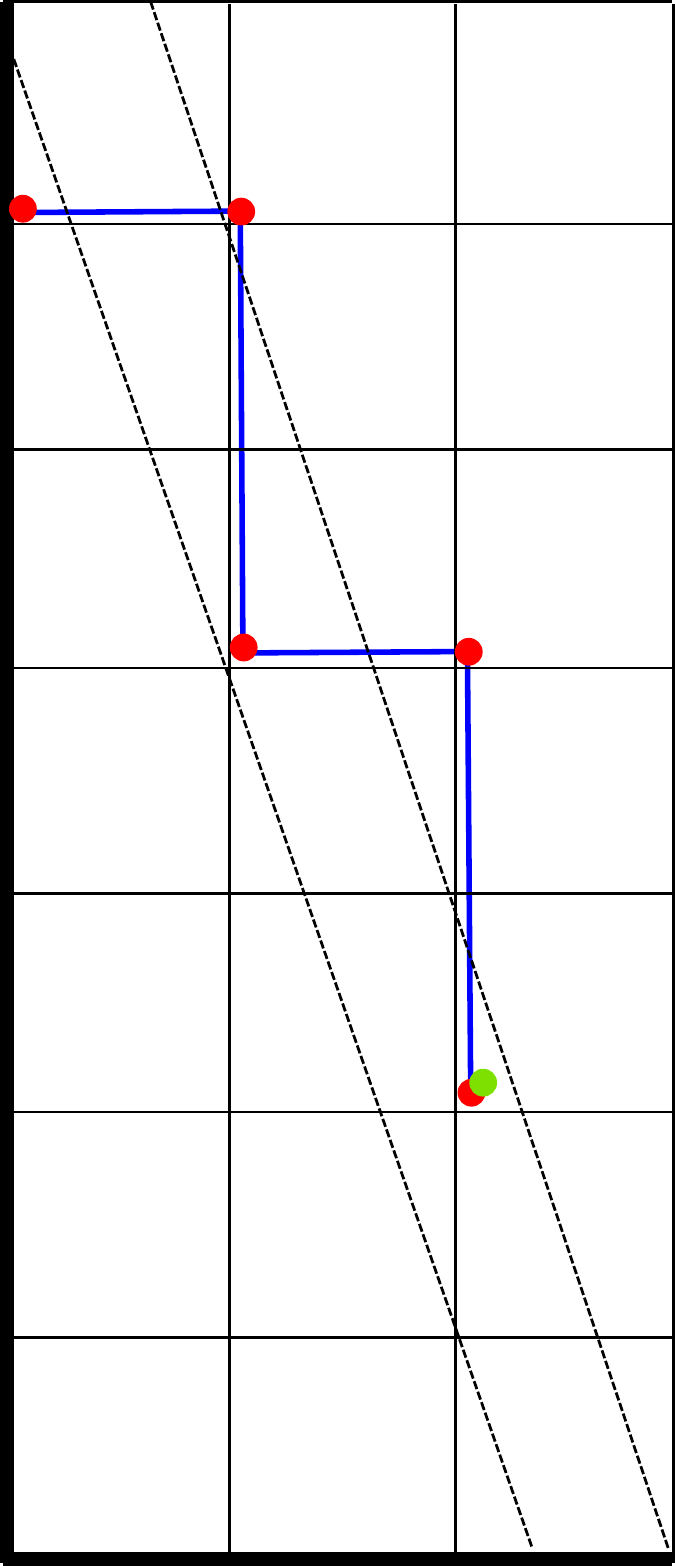}
\caption{$\icfk^\infty(T(3,7))$ simplified}
\label{up37}
\end{figure}The figure    includes two dashed lines of slope $-3$  corresponding to filtration levels when $t = 1/2$.  The lower line is the boundary of the region

$$ \frac{t}{2} \Min(x) + (1-\frac{t}{2} )\Max(x)\le \frac{7}{4}. $$
The upper line is the boundary of the region
$$ \frac{t}{2} \Min(x) + (1-\frac{t}{2} )\Max(x)\le \frac{9}{4}. $$

Continuing to work with $t  = 1/2$, the least value of $s$ so that the region  $\frac{t}{2} \Min(x) + (1-\frac{t}{2} )\Max(x)\le    s$ contains a generator of homology at grading 0 is $s = 3 /2$.  The   least value of $s$ so that the region  $\frac{t}{2} \Min(x) + (1-\frac{t}{2} )\Max(x)\le   s$ contains a generator of homology at grading 1 is $s = 2  $.  Thus,  $ \overline{\Upsilon}_{1/2}(K) = -3   $ and $ \underline{\Upsilon}_{1/2}(K) = -4$.

For general $t$ we have 

$$\overline {\Upsilon}_{t}(K) = 
\begin{cases}
-6t, &\text{\ if\ } 0\le  t\le \frac{2}{3}\\
-4, &\text{\ if\ }   \frac{2}{3}\le t \le 2,\\

\end{cases}$$

$$ \underline{\Upsilon}_{t}(K) = -4  \text{\ for all\ } t \in [0,2].$$


\section{Concordance invariance and knot inverses.}\label{sec:concordance}
\subsection{Concordance} 
\begin{theorem}
If $K$ and $J$ are concordant knots, then $ \overline {\Upsilon}_{t}(K) =   \overline {\Upsilon}_{t}(J)$ and $ \underline {\Upsilon}_{t}(K) =    \underline {\Upsilon}_{t}(J)$.

\end{theorem}

\begin{proof} A result of Hendricks and Hom~\cite{h_h} states that  if $L$ is a slice knot, then $\cfk^\infty(L)$ splits as the direct sum of involutive complexes:  $$\cfk^\infty(L) \cong \calt \oplus \cala,$$ where $\calt \cong \F[U,U^{-1}]$ and $\cala$ is acyclic.   (This result generalizes an analogous  result of Hom~\cite{hom2} that holds for noninvolutive complexes.  The proof depended on results of  Zemke~\cite{zemke} concerning the involutive homology of connected sums of knots.)      Thus, we can write 
$$   \cfk^\infty(K \cs -J) \cong  \calt_1  \oplus \cala_1$$ as involutive complexes, where $\calt_1 \cong \F[U,U^{-1}]$  and   $\cala_1$ is  acyclic.
It follows that  $$\cfk^\infty(K \cs {-J} \cs J) \cong \cfk^\infty(J) \oplus (\cala_1 \otimes \cfk^\infty(J))$$  and according to a connected sum formula for involutive homology given in~\cite{hmz}, this is again a direct sum of involutive complexes (but the involution on $\cala_1 \otimes \cfk^\infty(J)$ is not necessarily  the tensor product of the involutions; see~\cite{zemke} for details).  Since the second summand is acyclic, we write 
$$\cfk^\infty(K \cs {-J} \cs J) \cong \cfk^\infty(J) \oplus \cala_2.$$

Next, we   write 
$$\cfk^\infty(K \cs {-J} \cs J) \cong \cfk^\infty(K )\otimes \cfk^\infty( {-J} \cs J),$$ which can be rewritten as 
$$\cfk^\infty(K \cs {-J} \cs J) \cong \cfk^\infty(K )\otimes  (\calt_2  \oplus \cala_2),$$  since $J \cs -J$ is slice.  As before,   $\calt_2  \oplus \cala_2$ is a direct sum of involutive complexes,   
$$\cfk^\infty(K \cs {-J} \cs J) \cong \cfk^\infty(K ) \oplus   \cala_3.$$  In summary we have the following decompositions of involutive complexes:
$$\cfk^\infty(K) \oplus \cala_3\cong \cfk^\infty(J ) \oplus   \cala_2.$$ 

The acyclic summands do not affect the value of either $  \overline {\Upsilon}_{t}(K) $ or  $\underline {\Upsilon}_{t}(K)$, and thus the proof is complete. 
\end{proof}

\section{The knot concordance invariants $\overline{V}_0(K)$ and $\underline{V}_0(K)$}\label{sec:v0}
In~\cite{hm}, two knot concordance invariants $\overline{V}_0(K)$ and $\underline{V}_0(K)$ are defined.  These can be interpreted in terms of $\Upsilon$.

\begin{theorem}
$$\overline{V}_0(K) =  -\frac{1}{2}\overline{\Upsilon}_2(K),$$
$$\underline{V}_0(K) =  -\frac{1}{2}\underline{\Upsilon}_2(K).$$

\end{theorem}
\begin{proof}
Both $ \overline{V}_0(K)  $ and $\underline{V}_0(K)$  are defined in terms of the involutive correction terms  for large surgery on $K$:  $\overline{d}(S^3_p(K), [0])$ and 
$\underline{d}(S^3_p(K), [0])$.  These are computed in terms of  the maximal gradings of even and odd non-torsion elements in   the homology of $\icfk^\infty(K)_{\text{Max}\{i,j\} \le 0}$.    More precisely, they are minus one half these gradings.

Suppose the maximal grading of a (non-torsion) class in $\icfk^\infty(K)_{\text{Max}\{i,j\} \le 0}$ of even grading is $a$.  Then it follows from Formula 1 of~\cite{hm} that  $\overline{V}_0(K) = -a/2$.  As a consequence, the involutive complex  $\icfk^\infty(K)_{\text{Max}\{i,j\} \le s}$ contains a non-torsion class of grading 0 if and only if $s \ge -a/2$.  Thus  $\overline{\nu}_2(K) = -a/2$.  It now follows that $  \overline{\Upsilon}_2(K) =  a$, as desired.  A  similar argument works for the lower invariants.

\end{proof}

\section{Three-genus} \label{sec:three-genus} The three-genus bounds that arises from  $  \overline{\Upsilon}_K$ and $ \underline{\Upsilon}_K$ are almost immediate, following in the same way as   the lower bounds on $g_4(K)$ coming from $\Upsilon_K$:  $g_4(K) \ge \Upsilon'_K(t)$ for all $t\in [0,2]$ at which $\Upsilon_K(t)$ is nonsingular.  The  proof of this inequality   uses only the fact that $\cfk^\infty(K)$ is chain homotopy equivalent to complex for which all filtration levels satisfy $\big| \alg - \Alex \big| \le g_3(K)$ (see~\cite{oss}, or the expository account in~\cite{livingston}). The same constraint holds  for the involutive complex with the Max-Min filtration, so the same proof applies. 


\end{document}